\documentclass[a4paper,12pt]{article}

\usepackage{amsfonts}
\usepackage{amscd,color}
\usepackage{amsmath,amsfonts,amssymb,amscd}
\usepackage{indentfirst,graphicx,epsfig}
\usepackage{graphicx,psfrag}
\input{epsf}
\usepackage{graphicx}
\usepackage{epstopdf}
\usepackage{caption}

\newtheorem{thm}{Theorem}[section]

\newtheorem{lem}[thm]{Lemma}

\newtheorem{cor}[thm]{Corollary}

\newenvironment {proof} {\noindent{\em Proof.}}{\hspace*{\fill}$\Box$\par\vspace{4mm}}

\title{  More on the $k$-color connection number of a graph\footnote{Supported by NSFC No.11371205 and 11531011.}}

\author{Hong Chang, Zhong Huang, Xueliang Li \\
{\small  Center for Combinatorics and LPMC}\\
{\small Nankai University, Tianjin 300071, P.R. China}\\
{\small Email: changh@mail.nankai.edu.cn, 2120150001@mail.nankai.edu.cn, lxl@nankai.edu.cn}\\
}
\date{}
\begin{document}
\maketitle
\begin{abstract}
An edge-colored graph $G$ is $k$-color connected if, between each pair of vertices, there exists a path using at least $k$ different colors. The $k$-color connection number of $G$, denoted by $cc_{k}(G)$, is the minimum number of colors needed to color the edges of $G$ so that $G$ is $k$-color connected.  First, we prove that let $H$ be a subdivision of a connected graph $G$, then $cc_{k}(H)\leq cc_{k}(G)$. Second, we give sufficient conditions to guarantee that $cc_{k}(G)=k$ in terms of minimum degree and the number of edges for 2-connected graphs. As a byproduct, we show that almost all graphs have the $k$-color connection number $k$. At last, we investigate the relationship between the $k$-color connection number and the rainbow connection number for a connected graph. In addition, we give exact values of $k$-color connection numbers for some graph classes: subdivisions of the wheel and the complete graph, and the generalised $\theta$-graph. \\[2mm]
\end{abstract}

\section{Introduction}

All graphs in this paper are undirected, finite and simple. We follow \cite{BM} for graph theoretical notation and terminology not described here. Let $G$ be a graph, we use $V, E, n, m$, and $\delta$ to denote its vertex set, edge set, number of vertices, number of edges, and minimum degree of $G$, respectively. Given two graphs $G$ and $H$, the union of $G$ and $H$, denoted by $G\cup H$, is the graph with vertex set $V(G)\cup V(H)$ and edge set $E(G)\cup E(H)$. The join of $G$ and $H$, denoted by $G+H$, is obtained from $G\cup H$ by joining each vertex of $G$ to each vertex of $H$.

Let $G$ be a nontrivial connected graph with an associated \emph{edge-coloring} $c : E(G)\rightarrow \{1, 2, \ldots, t\}$, $t \in \mathbb{N}$, where adjacent edges may have the same color. If adjacent edges of $G$ are assigned different colors by $c$, then $c$ is a \emph{proper (edge-)coloring}. For a graph $G$, the minimum number of colors needed in a proper coloring of $G$ is referred to as the \emph{edge-chromatic number} of $G$ and denoted by $\chi'(G)$. A path of an edge-colored graph $G$ is said to be a \emph{rainbow path} if no two edges on the path have the same color. The graph $G$ is called \emph {rainbow connected} if for every pair of distinct vertices of $G$ is connected by a rainbow path in $G$.  An edge-coloring of a connected graph is a \emph{rainbow connecting coloring} if it makes the graph rainbow connected. This concept of rainbow connection of graphs was introduced by Chartrand et al.~\cite{CJMZ} in 2008. The \emph{rainbow connection number} $rc(G)$ of a connected graph $G$ is the smallest number of colors that are needed in order to make $G$ rainbow connected. The readers who are interested in this topic can see \cite{LSS,LS} for a survey.

Inspired by rainbow coloring and proper coloring in graphs, Andrews et al.~\cite{ALLZ} and Borozan et al.~\cite{BFGMMMT} introduced the concept of proper-path coloring. Let $G$ be a nontrivial connected graph with an edge-coloring. A path in $G$ is called a \emph{proper path} if no two adjacent edges of the path receive the same color. An edge-coloring $c$ of a connected graph $G$ is a \emph{proper-path coloring} if every pair of distinct vertices of $G$ are connected by a proper path in $G$. And if $k$ colors are used, then $c$ is called a \emph{proper-path $k$-coloring}. An edge-colored graph $G$ is \emph{proper connected} if any two vertices of $G$ are connected by a proper path. The minimum number of colors that are needed in order to make $G$ proper connected is called \emph{proper connection number} of $G$, denoted by \emph{$pc(G)$}. Let $G$ be a nontrivial connected graph of order $n$ and size $m$, then we have that $1\leq pc(G) \leq \min\{\chi'(G), rc(G)\}\leq m$. Furthermore, $pc(G)=1$ if and only if $G=K_n$ and $pc(G)=m$ if and only if $G=K_{1,m}$ as a star of size $m$. For more details, we refer to \cite{GLQ,LLZ,LWY} and a dynamic survey \cite{LC}.

In 2011, Caro and Yuster \cite{CY} introduced a natural counterpart question of rainbow connection coloring, which is called the monochromatic connection coloring. An edge-coloring of a connected graph is called a monochromatic connection coloring if there is a monochromatic path joining any two vertices. Let $mc(G)$ denote the maximum number of colors used in monochromatic connection coloring of a graph $G$, which is called the monochromatic connection number of $G$. They presented some lower and upper bounds for $mc(G)$. Later, Cai et al. in \cite{CLW} obtained two kinds of Erd$\ddot{o}$s-Gallai-type results for $mc(G)$.

Notice that neither the rainbow connection coloring nor the proper-path coloring can force a certain number of colors to be used in the required colored path between each pair of vertices. On the other hand, the monochromatic connection coloring allow only one color to be assigned in the needed colored path between each pair of vertices. To address this gap, Coll et al. introduced the notion of the $k$-color connectivity of a graph in \cite{CLMMP}. An edge-colored graph $G$ is $k$-color connected if, between each pair of vertices, there is a path using at least $k$ different colors. And this coloring is called $k$-color connection coloring. The $k$-color connection number of $G$, is denoted by $cc_{k}(G)$, is the minimum number of colors needed to color the edges of $G$ so that $G$ is $k$-color connected. In \cite{CLMMP}, they classified the graphs that achieve all possible $k$-color connection numbers when $k$ is small. In addition, they detailed the $k$-color connection numbers for various graphs classes.

It is trivial that $cc_{1}(G)=1$ for any connected graph $G$, therefore, we assume that $k\geq2$ throughout this paper. In order to guarantee that $cc_{k}(G)$ exists, $G$ can contain no bridges, in other words, $G$ is supposed to be 2-edge connected. Otherwise, assume that $G$ has a bridge $uv$, then the endvertices $u$ and $v$ of the bridge are not connected by a path using more than one color. Note that for a connected graph $G$, the natural lower bound of $k$-color connection number is that $cc_{k}(G)\geq k$, moreover, the bound is tight in view of the wheel and the complete graph. But the general upper bound of this parameter is a problem. In \cite{CLMMP}, they conjectured that if $cc_{k}(G)$ exists, then $cc_{k}(G)\leq2k-1$.

Let us give an overview of the rest of this paper. In Section 2, we show that let $H$ be a subdivision of a connected graph $G$, then $cc_{k}(H)\leq cc_{k}(G)$. In Section 3, we provide sufficient conditions to guarantee that $cc_{k}(G)=k$ in terms of minimum degree and the number of edges for 2-connected graphs. As a byproduct, we prove that almost all graphs have the $k$-color connection number $k$. In section 4, it is showed that for two given positive integers $a$ and $b$ with $b\geq a$ and $k\leq a\leq 2k-2$, there exists a connected graph such that $cc_{k}(G)=a$ and $rc(G)=b$.

At the very beginning, we state some fundamental results on $k$-color connectivity, which are used in the sequel.

\begin{thm}\label{thm1.1}\cite{CLMMP}
If $cc_{k}(C_n)$ exists, then
\begin{equation*}
cc_{k}(C_n)=
\left\{
  \begin{array}{ll}
    DNE & \hbox{ if  $n\leq 2k-2$,}\\
    2k-1 & \hbox{ if  $n=2k-1$, and}\\
    k & \hbox{ if $n\geq2k$.} \\
    \end{array}
\right.
\end{equation*}
\end{thm}

Let $C_{n}^{ct}$ denote a cycle $C$ of length $n$ with the addition of a chord between two vertices at distance $t$ on $C$, the following result for general chorded cycles is presented in \cite{CLMMP}.

\begin{thm}\label{thm1.2}
For all $k\geq3$ and $t\leq k-1$, $cc_{k}(C_{2k-1}^{ct})=2k-t$.
\end{thm}

Recall that a block of a connected graph is a vertex-maximal 2-connected subgraph.

\begin{lem}\label{lem1.3}
Given a connected graph $G$ consisting of blocks $B_1,B_2, \ldots, B_t$. If $cc_{k}(G)$ exists, then $cc_{k}(G)=\max_{i} cc_{k}(B_i)$. Moreover, if $cc_{k}(B_i)$ does not exist for some $i$, then $cc_{k}(G)$ also does not exist.
\end{lem}

By Lemma \ref{lem1.3}, determining k-color connection number of a graph is ascribed to that of a 2-connected graph. In \cite{CLMMP}, the authors derived the following results for 2-connected graphs.

\begin{thm}\label{thm1.4}
If $G$ be a 2-connected graph containing a subgraph $H$ with $cc_{k}(H)=\ell$, then $cc_{k}(G)\leq\ell$.
\end{thm}

Combing Theorem \ref{thm1.1} and Theorem \ref{thm1.4}, one can easily obtain the following corollary.

\begin{cor}\label{cor1.5}\cite{CLMMP}
If $G$ is 2-connected and contains a cycle of length at least $2k$, then $cc_{k}(G)=k$.
\end{cor}

In \cite{CLMMP}, the authors also investigated k-connection numbers of two special graph classes: the wheel $W_n$ and the complete graph $K_n$.

\begin{thm}\label{thm1.6}
For $1 \leq k\leq n$ and $n\geq 3$, $cc_{k}(W_n)=k$.
\end{thm}

\begin{cor}\label{cor1.7}
For $k\geq3$, for all $n\geq k+1$, $cc_{k}(K_n)=k$.
\end{cor}

\section{Subdivided graphs}

An edge $e=uv$ of a graph $G$ is said to be subdivided when it is deleted and replaced by the edges $uw$ and $vw$ for some vertex $w$ not in $G$. A subdivision of a graph is a graph that can be obtained from $G$ by a sequence of edge subdivisions. The notion of the subdivision plays the key role in our main results.

The following lemma will be preliminary and useful in our discussion.

\begin{lem}\label{lem2.1}
Let $G'$ be the subdivision of a graph $G$ by subdivided an edge, then $cc_{k}(G')\leq cc_{k}(G)$.
\end{lem}
\begin{proof}
Suppose that the subdivided edge $uv$ is replaced by the edges $uw$ and $vw$. Let $c$ be a $k$-color connection coloring of $G$ with $cc_{k}(G)$ colors. Next, we provide an edge-coloring $c'$ of $G'$ defined by
\begin{equation*}
c'(e)=
\left\{
\begin{array}{ll}
    c(e) & \hbox{ if  $e\notin\{uw,vw\}$,}\\
    c(uv) & \hbox{ if  $e\in\{uw,vw\}$.}\\
\end{array}
\right.
\end{equation*}

It remains to prove that $G'$ is $k$-color connected under this coloring $c'$. It is sufficient to show that for any two vertices $x,y$ of $G'$, there exists a path using at least $k$ different colors. If $w\notin \{x,y\}$, note that there exists a path $P_1$, between $x$ and $y$, using at least $k$ colors in $G$ under the coloring $c$. Thus, we only need to follow the path $P_1$ in $G'$ under the coloring $c'$, but if $uv\in E(P_1)$, then the needed path is $xP_1uwvP_1y$ or $xP_1vwuP_1y$. If $w\in \{x,y\}$, without loss of generality, assume that $x=w$ and $y\neq u$. It is obtained that there exists a path $P_2$, between $y$ and $u$, using at least $k$ colors in $G$ under the coloring $c$. Therefore, we can use the path $P_2\cup uw$ in $G'$ under the coloring $c'$. Consequently, $cc_{k}(G')\leq cc_{k}(G)$.
\end{proof}

The following result is an immediate consequence of Lemma \ref{lem2.1} by proceeding with a sequence of edge subdivisions.

\begin{thm}\label{thm2.2}
Let $H$ be a subdivision of a connected graph $G$, then $cc_{k}(H)\leq cc_{k}(G)$.
\end{thm}

It is well known that the wheel $W_n$ is defined as $C_n+K_1$, the join of $C_n$ and $K_1$, constructed by joining a new vertex to every vertex of $C_n$. We say that the wheel-like graph $W_n^t$ is a graph obtained from $C_n$ by joining $t$ ($t\leq n$) disjoint paths from a new vertex $v_0$ to $t$ distinct vertices of $C_n$ with the same end-vertex $v_0$.

\begin{cor}\label{cor2.4}
For $3\leq k \leq t \leq n$, $cc_{k}(W_n^t)=k$.
\end{cor}

\begin{proof}
 In fact, $W_n^t$ is a subdivision of $W_k$ for $3\leq k \leq t \leq n$. It follows that $cc_{k}(W_n^t)\leq cc_{k}(W_k)=k$ by Theorem \ref{thm1.6} and Theorem \ref{thm2.2}. The lower bound $cc_{k}(W_n^t)\geq k$ is trivial. Thus, $cc_{k}(W_n^t)=k$.
\end{proof}

Notice that a subdivision of $K_3$ is exactly $C_n$, the corresponding result is showed by Theorem \ref{thm1.1}. The following is the result of $k$-color connection number of a subdivision of $K_n$ for $n\geq4$, we omit the obvious proof.

\begin{cor}\label{cor2.3}
Let $G$ be a subdivision of $K_n$ of order $n\geq4$ with $n\geq k+1$, then $cc_{k}(G)=k$.
\end{cor}

\section{ $2$-connected graphs }

Lemma \ref{lem1.3} implies that determining k-color connection number of a graph is ascribed to that of a 2-connected graph. Thus, we are focus on 2-connected graphs in this section.

A graph $G$ of order $n$ is called pancyclic if it contains cycles of all lengths $r$ for $3\leq r \leq n$. The circumference of a graph $G$, denoted by $C(G)$, is the length of a longest cycle of $G$.

\begin{thm}\label{thm3.1} \cite{D}
Let $G$ is a $2$-connected graph of order $n$ and minimum $\delta$. Then $C(G)\geq \min\{n, 2\delta\}$.
\end{thm}

Dirac \cite{D} has also shown the following two results for 2-connected graphs.

\begin{thm}\label{thm3.2}
Let $G$ is a graph with $n$ vertices. If $\delta\geq\frac{n}{2}$, then $G$ has a Hamiltonian cycle.
\end{thm}

\begin{thm}\label{thm3.3}
Suppose $G=G(n,m)$ is Hamiltonian and $m\geq\frac{n^2}{4}$. Then either $G$ is pancyclic or else $G=K_{\frac{n}{2},\frac{n}{2}}$. In particular, if $G=G(n,m)$ is Hamiltonian and $m>\frac{n^2}{4}$ then $G$ is pancyclic.
\end{thm}

First, We give minimum degree condition for $k$-color connection number $k$ for 2-connected graphs.

\begin{thm}\label{thm3.4}
If $G$ is a $2$-connected graph of order $n\geq4$ with $\delta\geq k$, then $cc_{k}(G)=k$.
\end{thm}

\begin{proof}
Since $G$ is $2$-connected, it follows that $G$ has a cycle of length at least $\min\{n, 2\delta\}$ by Theorem \ref{thm3.1}. Suppose that $n\geq2\delta\geq2k$, we have that $G$ has a cycle of length at least $2k$. Applying Corollary \ref{cor1.5} to $G$, $cc_{k}(G)=k$. Next, we just need to consider the case that $n<2\delta$ in the following. Note that $\delta>\frac{n}{2}$, which implies $G$ is Hamiltonian and $m\geq\frac{n\delta}{2}>\frac{n^2}{4}$. It follows from Theorem \ref{thm3.3} that $G$ is pancyclic.  Thus, $G$ has a cycle $C_0$ of length $n-1$. Note that $\delta\geq k$, it means $n-1\geq k$. Let $v_0 \notin C_0$. Since $d(v_0)\geq\delta\geq k$, the wheel-like graph $W_{n-1}^k$ can be obtained from $C_0$ by joining $v_0$ to its $k$ neighbours on $C_0$. It follows from Corollary \ref{cor2.4} that $cc_{k}(W_{n-1}^k)=k$.  Thus, $cc_{k}(G)\leq cc_{k}(W_{n-1}^k)=k$ by Theorem \ref{thm1.4}.
\end{proof}

Note that the minimum degree condition in Theorem \ref{thm3.4} is best possible. Consider the complete graph $K_k$ with $\delta=k-1$,  but $cc_{k}(K_k)$ does not exist. Other examples are that if $k=3$, let $G \in \{ C_5, C_{5}^{c2} \}$, then $G$ has the property that $\delta=k-1$ and $cc_{k}(G)\geq k+1$.

In fact, if $G$ is $k$-connected, then $G$ is 2-connected and $\delta\geq k$, which implies the following straightforward result.

\begin{cor}\label{cor3.5}
If $G$ is a $k$-connected graph of order $n\geq4$, then $cc_{k}(G)=k$.
\end{cor}

The authors in \cite{BH} proved the following result.

\begin{thm}\label{thm3.9}
For every nonnegative integer $k$, almost all graphs are $k$-connected.
\end{thm}

From Corollary \ref{cor3.5} and Theorem \ref{thm3.9}, one can easily obtain the following result.

\begin{thm}\label{thm3.10}
Almost all graphs have the $k$-color connection number $k$.
\end{thm}

Next, we provide a sufficient condition to guarantee $cc_{k}(G)\leq 2k-1$ in terms of the number of edges $m$. The following lemma is an important ingredient in our proofs.

\begin{lem}\label{lem3.6}\cite{B}
If $G=G(n,m)$, $m\geq\frac{(c-1)(n-1)}{2}+1$, $3\leq c \leq n$, then the circumference $C(G)$ of $G$ is at least $c$.
\end{lem}

\begin{thm}\label{thm3.7}
If $G$ is a $2$-connected graph with $m\geq(k-1)(n-1)+1$, $n\geq2k-1$, then $cc_{k}(G)\leq 2k-1$.
\end{thm}

\begin{proof}
Since $m\geq(k-1)(n-1)+1$, it follows that $C(G)\geq 2k-1$ by Lemma \ref{lem3.6}. Let $C$ be a cycle of length $C(G)$, then $cc_{k}(C)\leq 2k-1$. With the aid of Theorem \ref{thm1.4}, we have that $cc_{k}(G)\leq cc_{k}(C)\leq 2k-1$.
\end{proof}

With the similar argument in Theorem \ref{thm3.7}, but we can obtain the following stronger result.

\begin{thm}\label{thm3.8}
If $G$ is a $2$-connected graph with $m\geq\frac{(2k-1)(n-1)}{2}+1$, $n\geq2k$, then $cc_{k}(G)=k$.
\end{thm}

\begin{proof}
Since $m\geq\frac{(2k-1)(n-1)}{2}+1$, it follows that $C(G)\geq 2k$ by Lemma \ref{lem3.6}. Thus, we have that $cc_{k}(G)=k$ by Corollary \ref{cor1.5}.
\end{proof}

\section{ Comparing $cc_{k}(G)$ and $rc(G)$ }

In this section, we investigate the relationship between the $k$-color connection number and the rainbow connection number for a connected graph, we first present two lemmas.

\begin{lem}\label{lem4.1}\cite{CJMZ}
For each integer $n\geq4$, then $rc(C_n)=\lceil \frac{n}{2}\rceil$.
\end{lem}

\begin{lem}\label{lem4.2}\cite{CJMZ}
For $n\geq3$, the rainbow connection number of the wheel $W_n$ is
\begin{equation*}
rc(G)=
\left\{
  \begin{array}{ll}
    1 & \hbox{ if $n=3$,}\\
    2 & \hbox{ if  $4\leq n\leq 6$,}\\
    3 & \hbox{ if $n\geq7$.} \\
    \end{array}
\right.
\end{equation*}
\end{lem}

Recall the fact that $cc_{k}(W_n)=k$ and $rc(W_n)\leq3$ for $3\leq k\leq n$, which implies that the $k$-color connection number is greater than the rainbow connection number for the wheel. However, note that $cc_{k}(C_n)=k$ and $rc(C_n)=\lceil \frac{n}{2}\rceil$ for any sufficiently large $n$, it follows that the cycle with very large length has the opposite property that the rainbow connection number is larger than the $k$-color connection number. These results naturally lead the following question: for two given positive integers $a$ and $b$, whether there exists a connected graph such that $cc_{k}(G)=a$ and $rc(C)=b$?

The generalised $\theta$-graph $G=Q_{q_1,q_2,\ldots,q_t}$ is defined as the union of $t\geq2$ paths $Q_1,Q_2,\ldots,Q_t$ with length $q_1\geq \ldots \geq q_t\geq1$, where $q_{t-1}\geq2$, and the paths are pairwise internally vertex-disjoint with the same two end-vertices. Note that $G$ is 2-connected, and the length of the longest cycle is $q_1+q_2$. In order to resolve the above question, we first study the $k$-color connection number and the rainbow connection number for the generalised $\theta$-graph.

\begin{thm}\label{thm4.1}
Let $G=Q_{q_1,q_2,\ldots,q_t}$ be a generalised $\theta$-graph, then
\begin{equation*}
cc_{k}(G)=
\left\{
  \begin{array}{ll}
    k & \hbox{ if  $q_1+q_2\geq2k$,}\\
    2k-1 & \hbox{ if  $q_1+q_2=2k-1$ and $t=2$,}\\
    q_1+1 & \hbox{ if $q_1+q_2=2k-1$ and $t\geq3$,} \\
    DNE & \hbox{ if $q_1+q_2\leq2k-2$.}\\
    \end{array}
\right.
\end{equation*}
\end{thm}

\begin{proof}
Let $Q_1\cup Q_2\cup\ldots\cup Q_t$ be the paths of $G$, and $x,y$ be their common end-vertices. Let $C=Q_1\cup Q_2=v_1v_2\ldots v_{q_1+q_2}$ denote the longest cycle of length $q_1+q_2$, without loss of generality, assume that $x=v_{k+1}$ and $y=v_{k+1-q_2}$. The result trivially holds for $q_1+q_2\geq2k$ by Corollary \ref{cor1.5}. Next, we assume that $q_1+q_2=2k-1$. Suppose that $t=2$, then $cc_{k}(G)=2k-1$ by Theorem \ref{thm1.1}. Suppose that $t\geq3$, consider the subgraph $H=Q_1\cup Q_2\cup Q_3$ of $G$. It is known that $H$ is also 2-connected, and a subdivision of $C_{2k-1}^{cq_2}$. Combining Theorem \ref{thm1.2} and Theorem \ref{thm2.2}, $cc_{k}(G)\leq cc_{k}(H)\leq cc_{k}(C_{2k-1}^{cq_2})=2k-q_2=q_1+1$. In order to show that $cc_{k}(G)\geq q_1+1$, without loss of generality, assume that $Q_1=v_{k+1-q_2}v_{k+q_2}\ldots v_1v_{2k-1}\ldots v_{k+1}$, and $Q_2=v_{k+1-q_2}v_{k+2-q_2}\ldots v_{k+1}$. We claim that no two edges of the form $v_iv_{i+1}$ for $1\leq i\leq 2k-q_2$ may share a color. Assume, to the contrary, that some pair of edges $v_iv_{i+1}$ and $v_jv_{j+1}$ share a color, where $1\leq i<j\leq 2k-q_2$. Suppose that $j\leq i+k-1$, the pair of vertices $v_i$ and $v_{i+k}$ are not connected by a path using $k$ colors. Suppose that $j>i+k-1$, the pair of vertices $v_j$ and $v_{j+k}$ are not connected by a path using $k$ colors. Thus, we have that $cc_{k}(G)\geq 2k-q_2=q_1+1$. At last, we may assume that $q_1+q_2\leq2k-2$, consider the pair of vertices $v_{k+1-q_2}$ and $v_{2k-q_2}$, there exists no path of length at least $k$ between them, so $G$ is not $k$-color connected.
\end{proof}

\begin{thm}\label{thm4.2}
Let $G=Q_{q_1,q_2,\ldots,q_t}$ be a generalised $\theta$-graph, then $rc(G)=\lceil \frac{q_1+\cdots+q_t}{2}\rceil$.
\end{thm}

\begin{proof}
Let $Q_1\cup Q_2\cup\ldots\cup Q_t$ be the paths of $G$, and $x,y$ be their common end-vertices. Consider each pair of paths $Q_i$ and $Q_j$ with $1\leq i\neq j\leq t$. In order to guarantee that any two vertices of $Q_i$ and $Q_j$ are connected by a rainbow path, the number of colors is needed to be at least $\lceil \frac{q_i+q_j}{2}\rceil$, it follows that $rc(G)\geq\lceil \frac{q_1+\cdots +q_t}{2}\rceil$. It remains to show that $rc(G)\leq\lceil \frac{q_1+\cdots +q_t}{2}\rceil$. Let $C_1\cup\cdots\cup C_{\lceil \frac{t}{2}\rceil}$ is a partition of t paths $Q_1,Q_2,\ldots,Q_t$, where $C_i$ is a cycle that consist of two paths for $1\leq i\leq \lfloor \frac{t}{2}\rfloor$. If $\lceil \frac{t}{2}\rceil=\lfloor \frac{t}{2}\rfloor+1$, then $C_{\lceil \frac{t}{2}\rceil}$ is a path. Take a partition of these $t$ paths, still denoted by $C_1\cup\cdots \cup C_{\lceil \frac{t}{2}\rceil}$, satisfying the number of cycles of even length is as possible as much. We define a coloring of $G$: color the edges of $C_i$ with $rc(C_i)$ colors such that $C_i$ is rainbow connected for $1\leq i\leq \lfloor \frac{t}{2}\rfloor$. If $\lceil \frac{t}{2}\rceil=\lfloor \frac{t}{2}\rfloor+1$, suppose that $C_{\lceil \frac{t}{2}\rceil}=v_1\ldots v_\ell$, then we color the edges $v_iv_{i+1}$ and $v_{\ell-1-i}v_{\ell-i}$ with color $i$ for $1\leq i\leq \lfloor \frac{\ell}{2}\rfloor$. Notice that two color sets that used in each pair of cycles $C_i$ and $C_j$ is completely distinct, which means this coloring uses $\lceil \frac{q_1+\cdots +q_t}{2}\rceil$. It is easy to check out that $G$ is rainbow connected under this coloring. Thus, $rc(G)\leq \lceil \frac{q_1+\cdots +q_t}{2}\rceil$. The proof is complete.

\end{proof}

With the help of Theorem \ref{thm4.1} and  Theorem \ref{thm4.2}, we obtain the following result.

\begin{thm}\label{thm4.3}
Let $a$ and $b$ be integers with $b\geq a$ and $k\leq a\leq 2k-2$. Then there exists a connected graph such that $cc_{k}(G)=a$ and $rc(G)=b$.
\end{thm}
\begin{proof}
Suppose that $G=Q_{q_1,q_2,\ldots,q_t}$ is a generalised $\theta$-graph. Let $Q_1\cup Q_2\cup\ldots\cup Q_t$ be the paths of $G$ with length $q_1\geq \ldots \geq q_t\geq1$, where $q_{t-1}\geq2$. If $a=k$, by Theorems \ref{thm4.1} and \ref{thm4.2}, $G$ is the required graph when $q_1+q_2\geq2a$, and $\lceil \frac{q_1+\cdots +q_t}{2}\rceil=b$. If $k+1\leq a \leq2k-2$, $G$ is the needed graph when $q_1=a-1$, $q_2=2k-a$, and $\lceil \frac{q_1+\cdots +q_t}{2}\rceil=b$, where $t\geq3$.
\end{proof}

{\bf Remark:} It is unsolved that whether there exists a connected graph such that $cc_{k}(G)=a$ and $rc(C)=b$ for $a\geq2k-1$ or $a>b$?

\end{document}